\documentclass[11pt]{article}
\usepackage[numbers,sort&compress]{natbib}
\usepackage{enumerate}
\usepackage{amscd}
\usepackage{amsmath}
\usepackage{latexsym}
\usepackage{amsfonts}
\usepackage{amssymb}
\usepackage{amsthm}
\usepackage{verbatim}
\usepackage{mathrsfs}
\usepackage{enumerate}
\usepackage{hyperref}

\theoremstyle{plain}\newtheorem{definition}{Definition}[section]
\theoremstyle{definition}\newtheorem{theorem}{Theorem}[section]
\theoremstyle{plain}\newtheorem{lemma}[theorem]{Lemma}
\theoremstyle{plain}\newtheorem{coro}[theorem]{Corollary}
\theoremstyle{plain}
\theoremstyle{remark}\newtheorem{remark}{Remark}[section]
\usepackage{xcolor}

\newcommand{\norm}[1]{\left\|#1\right\|}
\newcommand{\Div}{\mathrm{div}\,}
\newcommand{\B}{\Big}

\newcommand{\be}{\begin{equation}}
\newcommand{\ee}{\end{equation}}
 \newcommand{\ba}{\begin{aligned}}
 \newcommand{\ea}{\end{aligned}}

\newcommand{\fbxo}{\int_{B_{\varepsilon}(x)}\!\!\!\!\!\!\!\!\!\!\!\!\!\!\!\! -~\,\,~\,~\,}

\newcommand{\fbxoo}{\int_{B_{\varepsilon}(0)}\!\!\!\!\!\!\!\!\!\!\!\!\!\!\!\! -~\,~\,~\,~\,}

  \newcommand{\f}{\frac}
    
  \newcommand{\ben}{\begin{enumerate}}
   \newcommand{\een}{\end{enumerate}}

\newcommand{\Rmnum}[1]{\expandafter\@slowromancap\romannumeral #1@}

\allowdisplaybreaks

\numberwithin{equation}{section}
\begin{document}
\title{On the energy and helicity    conservation of the incompressible Euler equations   }
\author{Yanqing Wang\footnote{College of Mathematics and   Information Science, Zhengzhou University of Light Industry, Zhengzhou, Henan  450002,  P. R. China Email: wangyanqing20056@gmail.com}, ~\, Wei Wei\footnote{School of Mathematics and Center for Nonlinear Studies, Northwest University, Xi'an, Shaanxi 710127,  P. R. China  Email: ww5998198@126.com}, ~\, Gang Wu\footnote{School of Mathematical Sciences,  University of Chinese Academy of Sciences, Beijing 100049, P. R. China Email: wugang2011@ucas.ac.cn}~~~~and ~\,Yulin Ye\footnote{School of Mathematics and Statistics,
Henan University,
Kaifeng, 475004,
P. R. China. Email: ylye@vip.henu.edu.cn}
 }
\date{}
\maketitle
\begin{abstract}
In this paper, we are concerned with the  minimal regularity of weak solutions implying the law of balance for both energy and helicity in the incompressible Euler equations. In the spirit of recent works due to
Berselli \cite{[B]} and
Berselli-Georgiadis \cite{[BG]}, it is shown that the energy of weak solutions is  invariant if  $v\in L^{p}(0,T;B^{\f1p}_{\f{2p}{p-1},c(\mathbb{N})} )$ with $1<p\leq3$
 and the helicity is conserved if  $v\in L^{p}(0,T;B^{\f2p}_{\f{2p}{p-1},c(\mathbb{N})} )$ with $2<p\leq3 $
for both the periodic domain   and the
whole space, which generalizes the classical work of Cheskidov-Constantin-Friedlander-Shvydkoy in \cite{[CCFS]}.
This  indicates the role of the time integrability, spatial integrability and differential regularity of the velocity  in the   conserved quantities  of weak solutions of the ideal fluid.

  \end{abstract}
\noindent {\bf MSC(2020):}\quad   35Q35,  35Q86, 76D03\\\noindent
{\bf Keywords:}  Euler equations; Onsager conjecture;    energy conservation; helicity conservation 
\section{Introduction}
\label{intro}
\setcounter{section}{1}\setcounter{equation}{0}

The evolution of homogeneous inviscid incompressible flows on $\Omega$ is described by the following Euler equations
\begin{equation}\left\{\begin{aligned}\label{Euler}
&v_{t}+v\cdot \nabla v+\nabla \Pi=0,\\ &\text{div}\, v=0,\\
&v|_{t=0}=v_{0}(x),
\end{aligned}\right.\end{equation}
where $\Omega$ is either the whole space $\mathbb{R}^{d}$ or the periodic domain $\mathbb{T}^{d}$ with $d=2,3$.
The unknown vector $v=(v_{1},\cdots,v_{d})$ is velocity field and $\Pi$ represents the pressure.
 The
initial  velocity $v_0$ satisfies   $\text{div}\,v_0=0$.  The  vorticity $\omega=\text{curl\,}v$ of the velocity field is determined by \be
\omega_{t}+v\cdot\nabla\omega-\omega\cdot\nabla v=0,~~\text{div\,}\omega=0.\label{vorticityeq}
\ee
As we know, there exist only two conserved  quantities  of second order in the Euler equations: kinetic energy  conservation
\begin{equation}\label{eng}
\f12\int_{\Omega}|v(x,t)|^{2}dx=\f12\int_{\Omega}|v(x,0)|^{2}dx
\end{equation} and helicity preservation
\begin{equation}\label{he}
\int_{\Omega} \omega(x,t)\cdot v(x,t)  dx=\int_{\Omega} \omega(x,0)\cdot v(x,0)  dx.
\end{equation} These two conserved  quantities  attract a lot of  attentions and research of physicists and mathematicians(see e.g. \cite{[Frisch],[Eyink1],[DR],[Chkhetiani],[De Rosa],
[Moffatt],[MT],[Eyink],[CCFS],[CET],[WWY231],[Onsager]}). The celebrated Kolmogorov 4/5 law and Yaglom 4/3 law are in terms of energy of the Euler equations. The study of energy and helicity plays an important role in the theory of turbulence (see e.g. \cite{[Chkhetiani],[Frisch],[Moffatt],[MT]} and  references therein). The famous Onsager conjecture  raised in \cite{[Onsager]} is stated that the critical H\"older regularity of weak solutions keeping the energy is 1/3. A number of works have devoted to the
minimal regularity of weak solutions of the Euler system to preserve the energy(see e.g. \cite{[Eyink],[CET],[CCFS],[B],[BG],[FW2018],[BGSTW]} and  references therein).  To be specific, we   give a brief survey here.
The energy of weak solutions $v$ of the Euler equations  \eqref{Euler} is  invariant if
one of the following conditions is satisfied
 \begin{itemize}
\item Eyink \cite{[Eyink]}: $v\in C_{\ast}^{\alpha}, \alpha>1/3$;
\item Constantin-E-Titi \cite{[CET]} $v\in L^{3}(0,T; B^{\alpha}_{3,\infty}(\mathbb{T}^{3}))$ with $\alpha>1/3$.
\item Duchon and Robert:    \cite{[DR]}:
 $\int_{\mathbb{T}^{d}}|v(t,x+\xi)-v(t,x)|^3dx \leq C(t)|\xi|\sigma(|\xi|)$, where $C(t)\in L^1(0,T)$ and $\sigma(a)\rightarrow 0$ as $a\rightarrow 0$;
\item Cheskidov-Constantin-Friedlander-Shvydkoy   \cite{[CCFS]}:  $v\in L^{3} (0,T;\in B^{1/3}_{3,c(\mathbb{N})});$
       \item  Fjordholm-Wiedemann \cite{[FW2018]}: $v\in L^{3} (0,T;\underline{B}^{1/3 }_{3,VMO}
  );$
\item   Berselli   \cite{[B]}:   $v\in L^{\f1\alpha+\delta}(0,T;C^{\alpha}(\mathbb{T}^{3})$ or  $v\in L^{\f1\alpha}(0,T;C_{w}^{\alpha}(\mathbb{T}^{3}))$,  $\alpha\in(1/3,1), \delta>0$;

\item  Berselli-Georgiadis \cite{[BG]}:
 $v\in L^{\f1\alpha}(0,T; B^{\beta}_{\f{2}{1-\alpha},\infty}(\mathbb{T}^{3})),
 \alpha\in(1/3,1),1/3<\alpha<\beta<1;$
 $v\in L^{r}(0,T; W^{1,q}(\mathbb{T}^{3})),
 r>\f{5q}{5q-6},q>2.$
 \end{itemize}
The inclusion relations of the aforementioned   spaces are that,
for $\alpha >1/3$,
\be\label{includ}
C_{\ast}^{\alpha}(C_{w}^{\alpha}) \subseteq C^\alpha \subseteq B^\alpha_{3,\infty}\subseteq B^{\frac{1}{3}}_{3,c(\mathbb{N})}\subseteq \underline{B}^{\frac{1}{3}}_{3,VMO}\subseteq B^{\frac{1}{3}}_{3,\infty}.
\ee
Inspired by the recent results  \cite{[B],[BG]} and classical critical Onsager sufficient condition via $B^{1/3}_{3,c(\mathbb{N})}$ in \cite{[CCFS]}, two natural questions are weather  $v\in L^{\f1\alpha}(0,T; B^{\beta}_{\f{2}{1-\alpha},\infty})$ with $1/3<\alpha<\beta$   can be replaced by
$v\in L^{\f1\alpha}(0,T;B^{ \alpha}_{\f{2}{1-\alpha},c(\mathbb{N})}(\mathbb{T}^{3}))$ and $v\in L^{\f{5q}{5q-6}}(0,T; W^{1,q}(\mathbb{T}^{3}))$
guarantees the energy conservation of weak solutions in the Euler equations \eqref{Euler}. The first result of this paper is formulated as follows.
\begin{theorem}\label{the1.1} Let $v$ be a weak solution of the incompressible   Euler equations \eqref{Euler} on the whole spaces $\mathbb{R}^{d}$ for $d=2,3$.
Then for any $0\leq t\leq T$, the energy equality \eqref{eng} of weak solutions   is preserved provided one of the following condition is satisfied
 \begin{enumerate}[(1)]
 \item $v\in L^{p}(0,T;B^{\f1p}_{\f{2p}{p-1},c(\mathbb{N})}(\mathbb{R}^{d})),1< p\leq3;$
 \item $\nabla v\in L^{p}(0,T; L^{\f{2dp}{(d+2)(p-1)}}(\mathbb{R}^{d})),1< p\leq3.$
 \end{enumerate}
\end{theorem}
\begin{remark}The first part of this theorem is a generalization of  classical energy conservation criterion $v\in L^{3} (0,T; B^{1/3}_{3,c(\mathbb{N})}) $ showed by  Cheskidov-Constantin-Friedlander-Shvydkoy for the Euler equations in \cite{[CCFS]}.
\end{remark}
\begin{remark}
Actually, by means of interpolation method, the range of $p$ in the first part of this theorem as well as Theorem \ref{the1.3} is $[1,3]$. Indeed, the interpolation  allows us to avoid the finite difference  quotients of higher order in the definition of   Besov spaces.
\end{remark}

The key point in \cite{[B],[BG]} is to make full use of the kinetic energy in the   Constantin-E-Titi type   estimates concerning  mollifier kernel on the periodic domain $\mathbb{T}^{d}$.
In the spirit of this and the study of energy equality for the Navier-Stokes equations  in \cite{[CL]},
the starting point of this theorem is the following interpolation inequality
$$2^{\f13 j}\|\Delta_{j}v\|_{L^{3}(\mathbb{R}^{d})}\leq \|\Delta_{j}v\|_{L^{2}(\mathbb{R}^{d})}^{1-\f{p}{3}}
\B[2^{j\f1p}\|\Delta_{j}v\|_{L^{\f{2p}{p-1}}(\mathbb{R}^{d})}\B]^{\f{p}{3}},$$
which helps us to reduce the first part of this theorem to the classical Onsager energy conservation sufficient condition  $v\in B^{1/3}_{3,c(\mathbb{N})} $
due to
Cheskidov-Constantin-Friedlander-Shvydkoy  in \cite{[CCFS]}.
In addition, it is worth pointing out that there exists  an alternative  proof to this part via Littlewood-Paley theory as \cite{[CCFS]}.
 Roughly speaking, it suffices to replace $\omega$ by $v$ and set  $k=p,\ell=q$ in \eqref{4.11}-\eqref{4.13}, hence we omit the details here.
To the knowledge of the authors, the statement similar to Theorem  \ref{the1.1}  does not exist in known literature.  Enlightening by the recent sufficient conditions for the energy equality of the Navier-Stokes equations in \cite{[WMH],[WY]},
 we will show the energy conservation  criteria based on the  gradient  of the veolcity for the weak solutions of the Euler equations.

For the torus case, we have the following result
\begin{theorem}\label{the1.2} Let $v$ be a weak solution to  the incompressible   Euler equations \eqref{Euler} on $\mathbb{T}^{d}$ with $d=2,3$. Then for any $0\leq t\leq T$, the energy equality \eqref{eng} of weak solutions   is preserved provided that  one of the following conditions is satisfied
 \begin{enumerate}[(1)]
  \item  $v\in L^{p}(0,T;B^{\f1p}_{\f{2p}{p-1},c(\mathbb{N})}(\mathbb{T}^{d})),1<p\leq3;$
 \item  $v\in L^{p}(0,T;\underline{B}^{\f1p}_{\f{2p}{p-1},VMO}(\mathbb{T}^{d})),1<p\leq3;$
\item $\nabla v\in L^{p}(0,T; L^{\f{2dp}{(d+2)(p-1)}}(\mathbb{T}^{d})),1<p\leq3.$
 \end{enumerate}
\end{theorem}
\begin{remark}The  criteria based on the gradient of the velocity here implies that
 $v\in L^{\f{5q}{5q-6}}(0,T; W^{1,q}(\mathbb{T}^3)) \ \text{with}\  q \geq 9/5$   guarantees the energy conservation.  Hence, this improves the recent energy conservation sufficient condition in \cite{[BG]}.
\end{remark}
\begin{remark}
According to inclusion relations \eqref{includ}, this theorem is an improvement of corresponding results by Berselli-Georgiadis in \cite{[BG]}.
\end{remark}
  As a byproduct, there holds
\begin{coro}\label{coro1.3}
Let $\Omega$ be the whole space $\mathbb{R}^{3}$ or the periodic domain $\mathbb{T}^{3}$. Suppose that $v$ is a Leray-Hopf weak solutions of the 3D Navier-Stokes on $\Omega$. Then the energy equality,  for $t\in[0,T)$,
\be\label{EI}
 \|v(t)\|_{L^{2}(\Omega)}^{2}+2 \int_{0}^{t }\|\nabla v\|_{L^{2}(\Omega)}^{2}ds= \|v_0\|_{L^{2}(\Omega)}^{2}.
\ee
is valid if
\be \label{NSc}
v\in L^{p}(0,T;B^{\f1p}_{\f{2p}{p-1},\infty}(\Omega)),1<p\leq3.
\ee
\end{coro}
\begin{remark}
 It is worth remarking that the analogous result in this corollary can be found in \cite{[CL]}. The novel here is the following  observation \be\label{sembed}
\|u\|_{ L^{p}(0,T;B^{\f1p}_{\f{2p}{p-1},\infty}(\mathbb{R}^{3}))}\leq C\|\nabla u\|_{ L^{p}(0,T; L^{\f{6p}{5p-5}}(\mathbb{R}^{3}))}, 1<p\leq3,
\ee from which we know that this corollary immediately yields the energy equality sufficient condition in terms of gradient of the velocity in  space $L^{p}(0,T; L^{\f{6p}{5p-5}}(\mathbb{R}^{3}))$  derived in \cite{[BC],[WMH],[Zhang],[BY2],[Berselli]} for the 3D Navier-Stokes equations.
\end{remark}

Next, we turn our attention to the helicity conservation \eqref{he}  of the Euler equations. We briefly list the known results in this direction as follows:
the weak solutions $v$ of the 3D Euler equations conserve the helicity if one of the following conditions is satisfied
  \begin{itemize}
\item Chae \cite{[Chae]}:  $\omega\in L^{3}(0,T;B^{\alpha}_{\f95,\infty}) $ with $\alpha>\f13;$
\item Chae \cite{[Chae1]}:  $v \in L^{r_{1}}(0,T;\dot{B}^{\alpha}_{\f92,q})$ and
$\omega \in  L^{r_{2}}(0,T;\in \dot{B}^{\alpha}_{\f95,q})$, with $\alpha>\f13$, $q\in [2,\infty]$, $r_1\in [2,\infty]$, $r_2\in [1,\infty]$ and $\frac{2}{r_{1}}+\frac{1}{r_{2}}=1;$
\item Cheskidov-Constantin-Friedlander-Shvydkoy   \cite{[CCFS]}: $v\in L^{3}(0,T;B^{\f23}_{3,c(\mathbb{N})});$
\item De Rosa \cite{[De Rosa]}:  $v\in L^{2r}(0,T;W^{\theta, 2p})$ and  $ \omega\in L^{\kappa}(0,T;W^{\alpha, q})$ with  $\f{1}{p}+\f1q=\f{1}{r}+\f{1}{\kappa}$  and $2\theta+\alpha\geq1$;
\item \cite{[WWY22]}: $ v\in L^{k} (0,T;\dot{B}^{\alpha}_{p,c(\mathbb{N})} )\ \text{and}\  \omega\in L^{ \ell  } (0,T; \dot{B}^{\beta}_{q,\infty})$ with $\f2k+\f1\ell=1,\f2p+\f1q =1,2\alpha+\beta\geq1.$
\end{itemize}
For the helicity conservation of weak solutions of    the Euler equations, our results read as follows:
\begin{theorem}\label{the1.3} Let $ v$ be a  weak solution of  incompressible Euler equations \eqref{Euler} in the sense of Definition \ref{eulerdefi}  and $\omega\in C([0,T];L^{\f{3}{2}}(\mathbb{R}^{3}))$. Then for any $t\in [0,T]$,  the helicity  conservation  \eqref{he} holds provided that one of the following conditions is
satisfied
 \begin{enumerate}[(1)]
 \item $ v\in L^{p}(0,T;B^{\f2p}_{\f{2p}{p-1},
     c(\mathbb{N})}(\mathbb{R}^{3})),2<p\leq3;$
    \item $ \nabla v\in L^{p}(0,T;L^{\f{6p}{5p-7}}(\mathbb{R}^{3})), 2<p\leq3;  $
 \item $ v\in L^{k} (0,T; {B}^{\alpha}_{\ell,c(\mathbb{N})} )\ \text{and}\  \omega\in L^{   p  } (0,T; {B}^{\beta}_{q,\infty})$ with $\f\theta k+\f1 p=1,\f\theta\ell+\f1q =\f{\theta}{2},\theta\alpha+\beta\geq1;$
\item $ v\in L^{k} (0,T; {B}^{\alpha}_{\ell ,\infty} )\ \text{and}\  \omega\in L^{p } (0,T;  {B}^{\beta}_{q,c(\mathbb{N})})$ with $\f\theta k+\f1 p=1,\f\theta\ell+\f1q =\f{\theta}{2},\theta\alpha+\beta\geq1. $
     \end{enumerate}
\end{theorem}
\begin{theorem}\label{the1.4} Let $ v$ be a  weak solution of  incompressible Euler equations \eqref{Euler} in the sense of Definition \ref{eulerdefi}  and $\omega\in C([0,T];L^{\f{3}{2}}(\mathbb{T}^{3}))$. Then for any $t\in [0,T]$,  the helicity  conservation  \eqref{he} holds provided that one of the following conditions is
	satisfied
 \begin{enumerate}[(1)]
   \item $ v\in L^{p}(0,T;B^{\f2p}_{\f{2p}{p-1},c(\mathbb{N})}
       (\mathbb{T}^{3})),2<p\leq3;$
   \item $ v\in L^{p}(0,T;\underline{B}^{\f2p}_{\f{2p}{p-1},VMO}(\mathbb{T}^{3})),
       2<p\leq3;$
   \item$ \nabla v\in L^{p}(0,T;L^{\f{6p}{5p-7}}(\mathbb{T}^{3})),2<p\leq3;$
 \item $ v\in L^{k} (0,T; {B}^{\alpha}_{\ell,c(\mathbb{N})} )\ \text{and}\  \omega\in L^{   p  } (0,T;  {B}^{\beta}_{q,\infty})$ with $\f\theta k+\f1 p=1,\f\theta\ell+\f1q =\f{\theta}{2},\theta\alpha+\beta\geq1;$
\item $ v\in L^{k} (0,T; {B}^{\alpha}_{\ell ,\infty} )\ \text{and}\  \omega\in L^{p } (0,T; {B}^{\beta}_{q,c(\mathbb{N})})$ with $\f\theta k+\f1 p=1,\f\theta\ell+\f1q =\f{\theta}{2},\theta\alpha+\beta\geq1. $
     \end{enumerate}
\end{theorem}
\begin{remark}
These theorems   extend the known
sufficient class for implying the helicity of weak solutions in the Euler equations.
\end{remark}
There may be potential applications of the strategy of this paper. It is useful to study  physical  conserved quantities in other hydrodynamic equations. A possible candidate is the following EMHD equations in \cite{[Galtier],[WYY]}
\be\left\{\ba\label{E-MHD}
&h_{t}+ \nabla\times\B[(\nabla\times h)\times h\B]=0, \\
 & \Div h=0.
 \ea\right.\ee
It should be pointed out that the progress of  the negative part of Onsager conjecture can be found in \cite{[Isett],[DS0],[DS1]}.

The rest of the paper is organized as follows.
 In Section 2, we present some notations such as the definitions of  various spaces and  auxiliary lemmas which will be  frequently used throughout this paper. A kind of  the Constantin-E-
Titi type estimates concerning mollifier kernel
are established to allow us to  use the kinetic energy.
Section 3 is devoted to the study of energy conservation of weak solutions to ideal fluid for both the whole space and the periodic case.  We are concerned with the conservation law of
helicity of weak solutions to invicid flow in Section 4.   Eventually, in Section 5, we present some concluding remarks.

\section{Notations and some auxiliary lemmas} \label{section2}

{\bf Sobolev spaces:} First, we introduce some notations used in this paper.
 For $p\in [1,\,\infty]$, the notation $L^{p}(0,\,T;X)$ stands for the set of measurable functions on the interval $(0,\,T)$ with values in $X$ and $\|f(t,\cdot)\|_{X}$ belonging to $L^{p}(0,\,T)$. The classical Sobolev space $W^{k,p}(\Omega)$ is equipped with the norm $\|f\|_{W^{k,p}(\Omega)}=\sum\limits_{|\alpha| =0}^{k}\|D^{\alpha}f\|_{L^{p}(\Omega)}$, where $\Omega$ is either the whole space $\mathbb{R}^{d}$ or the periodic domain $\mathbb{T}^{d}$ with $d=2,3$. \\
  {\bf Besov spaces:} We denote $\mathcal{S}$  the Schwartz class of rapidly decreasing functions, $\mathcal{S}'$ the
space of tempered distributions  and $\mathcal{S}'/\mathcal{P}$ the quotient space of tempered distributions which modulo polynomials.
  We use $\mathcal{F}f$ or $\widehat{f}$ to denote the Fourier transform of a tempered distribution $f$.
To define Besov  spaces, we need the following dyadic unity partition(see e.g. \cite{[BCD]}). Choose two nonnegative radial
functions $\varrho$, $\varphi\in C^{\infty}(\mathbb{R}^{d})$
supported respectively in the ball $\mathcal{B}=\{\xi\in
\mathbb{R}^{d}:|\xi|\leq \frac{4}{3} \}$ and the shell $\mathcal{C}=\{\xi\in
\mathbb{R}^{d}: \frac{3}{4}\leq |\xi|\leq
  \frac{8}{3} \}$ such that
\begin{equation*}
 \varrho(\xi)+\sum_{j\geq 0}\varphi(2^{-j}\xi)=1, \quad
 \forall\xi\in\mathbb{R}^{d}; \qquad
 \sum_{j\in \mathbb{Z}}\varphi(2^{-j}\xi)=1, \quad \forall\xi\neq 0.
\end{equation*}
Write $h=\mathcal{F}^{-1} \varphi $ and $\tilde{h}=\mathcal{F}^{-1}\varrho$, then nonhomogeneous dyadic blocks  $\Delta_{j}$ are defined by
$$
\Delta_{j} u:=0 ~~ \text{if} ~~ j \leq-2, ~~ \Delta_{-1} u:=\varrho(D) u =\int_{\mathbb{R}^d}\tilde{h}(y)u(x-y)dy,$$
$$\text{and}~~\Delta_{j} u:=\varphi\left(2^{-j} D\right) u=2^{jd}\int_{\mathbb{R}^d}h(2^{j}y)u(x-y)dy  ~~\text{if}~~ j \geq 0.
$$
The nonhomogeneous low-frequency cut-off operator $S_j$ is defined by
$$
S_{j}u:= \sum_{k\leq j-1}\Delta_{k}u=\varrho(2^{-j}D)u=2^{jd}\int_{\mathbb{R}^d}\tilde{h}(2^{j}y)u(x-y)dy, ~j\in \mathbb{N}\cup {0}.$$
Analogously, the homogeneous dyadic blocks $\dot{\Delta}_{j}$ and homogeneous low-frequency cut-off operators $\dot{S}_j$ are  defined  for $ \forall j\in\mathbb{Z}$ by
\begin{equation*}
  \dot{\Delta}_{j}u:= \varphi(2^{-j}D)u=2^{jd}\int_{\mathbb{R}^d}h(2^{j}y)u(x-y)dy,~j\in \mathbb{Z}
\end{equation*}
$$ \text { and }~~ \dot{S}_{j}u:=\varrho(2^{-j}D)u=2^{jd}\int_{\mathbb{R}^d}\tilde{h}(2^{j}y)u(x-y)dy,~j\in \mathbb{Z}$$
Now we introduce the definition of Besov spaces. Let $(p, r) \in[1, \infty]^{2}, s \in \mathbb{R}$, the nonhomogeneous Besov space
$$
B_{p, r}^{s}:=\left\{f \in \mathcal{S}^{\prime}\left(\mathbb{R}^{d}\right) ;\|f\|_{B_{p, r}^{s}}:=\left\|2^{j s}\right\| \Delta_{j} f\left\|_{L^{p}}\right\|_{\ell^{r}(\mathbb{Z})}<\infty\right\}
$$
and the homogeneous space
$$
\dot{B}_{p, r}^{s}:=\left\{f \in \mathcal{S}^{\prime}\left(\mathbb{R}^{d}\right) / \mathcal{P}\left(\mathbb{R}^{d}\right) ;\|f\|_{\dot{B}_{p, r}^{s}}:=\left\|2^{j s}\right\| \dot{\Delta}_{j} f\left\|_{L^{p}}\right\|_{\ell^{r}(\mathbb{Z})}<\infty\right\} .
$$
Moreover, for $s>0$ and $1\leq p,q\leq \infty$, we may write the
equivalent norm below in
  the nonhomogeneous Besov norm $\norm{f}_{B^s_{p,q}}$ of $f\in \mathcal{S}^{'}$ as
$$\norm{f}_{B^s_{p,q}}=\norm{f}_{{L^p}}+\norm{f}_{\dot{B}^s_{p,q}}.$$
Motivated by \cite{[CCFS]}, we define $ {B}^\alpha _{p,c(\mathbb{N})}$ to be the class of all tempered distributions $f$ for which
\begin{equation}\label{2.1}
\norm{f}_{ {B}^\alpha _{p,\infty}}<\infty~ \text{and}~ 	\lim_{j\rightarrow \infty} 2^{j\alpha}\norm{ {\Delta}_j f}_{L^p}=0,~~\text{for any}~1\leq p\leq \infty.
\end{equation}

A function $f$   is said to be in
$L^{p}(0,T;\dot{B}^\alpha_{q,c(\mathbb{N})}(\Omega))$  if
$$
\lim_{z\rightarrow0}\f{\left(\int_{0}^{T}\B(\int_{\Omega}|f(x+z,t)-f(x,t)|^{q}dx\B)^{\f{p}{q}}dt\right)^{\frac{1}{p}}}{|z|^{\alpha}}=0.
$$
For $0<\alpha<1$, it is worth pointing out that
$$
\lim_{z\rightarrow0}\f{\left(\int_{0}^{T}\B(\int_{\Omega}
|f(x+z,t)-f(x,t)|^{p}dx\B)^{\f{p}{q}}dt\right)^{\frac{1}{p}}}{|z|^{\alpha}}=0
\Leftrightarrow f\in L^{p}(0,T;\dot{B}^\alpha_{p,c(\mathbb{N})}(\Omega).$$

 A function $f$ belongs to
the Besov-VMO space  $L^p(0,T;\underline{B}^{\alpha}_{q,VMO}(\mathbb{T}^d))$  if it satisfies
$$\|f\|_{L^p(0,T;L^q(\mathbb{T}^d))}<\infty,$$
and
$$\ba
&\lim_{\varepsilon\rightarrow0}\f{1}{\varepsilon^{\alpha}}\left(\int_0^T\B[\int_{\mathbb{T}^d} \fbxo|f(x)-f(y)|^{q}dydx \B]^{\f{p}{q}}dt\right)^{\f1p}\\
=&\lim_{\varepsilon\rightarrow0}\f{1}{\varepsilon^{\alpha}}\left(\int_0^T\B[\int_{\mathbb{T}^d} \fbxoo|f(x)-f(x-y)|^{q}dydx \B]^{\f{p}{q}}dt\right)^{\f{1}{p}}=0.
\ea$$

{\bf Mollifier kernel:} Let $\eta_{\varepsilon}:\mathbb{R}^{d}\rightarrow \mathbb{R}$ be a standard mollifier.i.e. $\eta(x)=C_0e^{-\frac{1}{1-|x|^2}}$ for $|x|<1$ and $\eta(x)=0$ for $|x|\geq 1$, where $C_0$ is a constant such that $\int_{\mathbb{R}^d}\eta (x) dx=1$. For $\varepsilon>0$, we define the rescaled mollifier $\eta_{\varepsilon}(x)=\frac{1}{\varepsilon^d}\eta(\frac{x}{\varepsilon})$ and for  any function $f\in L^1_{loc}(\mathbb{R}^d)$, its mollified version is defined as
$$f^\varepsilon(x)=(f*\eta_{\varepsilon})(x)=\int_{\mathbb{R}^d}f(x-y)\eta_{\varepsilon}(y)dy,\ \ x\in \mathbb{R}^d.$$
Next, we collect some Lemmas which will be used in the present paper.
\begin{lemma}(Bernstein inequality \cite{[BCD]})\label{berinequ}  Let $\mathcal{B}$ be a ball of $\mathbb{R}^{d}$, and $\mathcal{C}$ be a ring of $\mathbb{R}^{d}$. There exists a positive constant $C$ such that for all integer $k \geq 0$, all $1 \leq a \leq b \leq \infty$ and $u \in L^{a}\left(\mathbb{R}^{d}\right)$, the following estimates are satisfied:
$$
\begin{gathered}
\sup _{|\alpha|=k}\left\|\partial^{\alpha} u\right\|_{L^{b}\left(\mathbb{R}^{d}\right)} \leq C^{k+1} \lambda^{k+d\left(\frac{1}{a}-\frac{1}{b}\right)}\|u\|_{L^{a}\left(\mathbb{R}^{d}\right)}, \quad \operatorname{supp} \hat{u} \subset \lambda \mathcal{B}, \\
C^{-(k+1)} \lambda^{k}\|u\|_{L^{a}\left(\mathbb{R}^{d}\right)} \leq \sup _{|\alpha|=k}\left\|\partial^{\alpha} u\right\|_{L^{a}\left(\mathbb{R}^{d}\right)} \leq C^{k+1} \lambda^{k}\|u\|_{L^{a}\left(\mathbb{R}^{d}\right)}, \quad \operatorname{supp} \hat{u} \subset \lambda \mathcal{C}.
\end{gathered}
$$
\end{lemma}
This lemma yields that if $1\leq p_{1}\leq p_{2} \leq\infty,1\leq r_{1}\leq r_{2} \leq\infty$, there holds
\be\label{besovinclu}
B^{s}_{p_{1},r_{1}}\subseteq B^{s-d(\f1p_{1}-\f1p_{2})}_{p_{2},r_{2}}
\ee
\begin{lemma}(\cite{[WWY22],[Chae],[Yu]})\label{lem2.2}
Let $\Omega$ denote the whole space $\mathbb{R}^{d}$ or the periodic domain $\mathbb{T}^{d}$.
 Suppose that  $f\in L^p(0,T; {B}^\alpha_{q,\infty}(\Omega))$, $g\in L^p(0,T;{B}^\beta_{q,c(\mathbb{N})}(\Omega))$ with $\alpha, \beta\in (0,1)$,  $ p,q\in [1,\infty]$,   then there holds that,  for any $k\in \mathbb{N}^+$, as $\varepsilon\rightarrow0,$
 \begin{enumerate}[(1)]
 \item $ \|f^{\varepsilon} -f \|_{L^{p}(0,T;L^{q}(\Omega))}\leq \text{O}(\varepsilon^{\alpha})\|f\|_{L^p(0,T;\dot{B}^\alpha_{q,\infty}(\Omega))}$;
   \item   $ \|\nabla^{k}f^{\varepsilon}  \|_{L^{p}(0,T;L^{q}(\Omega))}\leq \text{O}(\varepsilon^{\alpha-k})\|f\|_{L^p(0,T;\dot{B}^\alpha_{q,\infty}(\Omega))}$;
       \item $ \|g^{\varepsilon} -g \|_{L^{p}(0,T;L^{q}(\Omega))}\leq \text{o}(\varepsilon^{\beta})\|g\|_{L^p(0,T;\dot{B}^\beta_{q,c(\mathbb{N})}(\Omega))}$;
   \item   $ \|\nabla^{k}g^{\varepsilon}  \|_{L^{p}(0,T;L^{q}(\Omega))}\leq \text{o}(\varepsilon^{\beta-k})\|g
       \|_{L^p(0,T;\dot{B}^\beta_{q,c(\mathbb{N})}(\Omega))}$.
 \end{enumerate}
\end{lemma}
\begin{remark}\label{rem2.1}
The results still hold for $g\in L^p(0,T;\underline{B}^\beta_{q,VMO}(\mathbb{T}^{d}))$, whose proof is proposed  in \cite{[WYY],[BGSTW]}.
\end{remark}
Next, we will modify the Constantin-E-Titi type commutator estimates to allow us to take full use of  kinetic energy.
\begin{lemma} 	\label{lem2.3}
Let $\Omega$ denote the whole space $\mathbb{R}^{d}$ or the periodic domain $\mathbb{T}^{d}$.
	Assume that $0<\alpha <1$,$0<\theta\leq2$, $1\leq r,s,p_{1},p_{2}\leq\infty$, $\frac{1}{r}=\frac{1}{p_1}+\frac{1}{p_2}$ and $\frac{1}{s}=\frac{1}{q_1}+\frac{1}{q_2}$,
then as $\varepsilon\rightarrow0,$ there holds,
	\begin{align} \label{cet}
		\|(ff)^{\varepsilon}- f^{\varepsilon}f^{\varepsilon}\|_{L^r(0,T;L^s(\Omega))} \leq C\text{o}(\varepsilon^{\theta\alpha }),	
	\end{align}
if one of the following conditions is satisfied   \begin{enumerate}[(1)]
 \item
 $f\in L^{(2-\theta)p_{2}}(0,T;L^{(2-\theta)q_{2}}(\Omega))$ and $f\in L^{\theta p_{1}}(0,T;B_{\theta q_{1},c(\mathbb{N})}^{\alpha}(\Omega));$
  \item $f\in L^{(2-\theta)p_{2}}(0,T;L^{(2-\theta)q_{2}}(\Omega))$ and  $f\in L^{\theta p_{1}}(0,T;\underline{B}_{\theta q_{1},VMO}^{\alpha}(\Omega))$,$q_{2}\geq \f{ q_{1}}{q_{1}-1}$, $p_{2}\geq\f{q_{1}}{q_{1}-1}$ .
\end{enumerate}\end{lemma}

 \begin{proof}
First, we recall the following   identity observed  by Constantin-E-Titi   in \cite{[CET]} that
	\be\label{CETI}\ba&(ff)^{\varepsilon}(x)- f^{\varepsilon}f^{\varepsilon}(x)\\
=&
\int_{\Omega}\eta_{\varepsilon}(y)
[f(x-y)-f(x)][f(x-y)-f(x)]dy-
(f-f^{\varepsilon})(f-f^{\varepsilon})(x).
	\ea\ee

(1)	Using the H\"older inequality and the Minkowski inequality and the definition of $B_{\theta q_{1},c(\mathbb{N})}^{\alpha}$, we obtain
	$$\begin{aligned} & \|(ff)^{\varepsilon}- f^{\varepsilon}f^{\varepsilon}\|_{L^r(0,T;L^s(\Omega))}\\ \leq&C\int_{|y|\leq\varepsilon}\eta_{\varepsilon}(y)\| f(\cdot-y)-f(\cdot)\|^{\theta}_{L^{p_1\theta}(0,T;L^{q_1\theta}(\Omega))} \| f(\cdot-y)-f(\cdot)\|^{2-\theta}_{L^{p_2(2-\theta)}(0,T;L^{q_2(2-\theta)}(\Omega))}dy \\&+C\| f-f^{\varepsilon}\|^{\theta}_{L^{p_1\theta}(0,T;L^{q_1\theta}(\Omega))}  \| f-f^{\varepsilon}\|^{2-\theta}_{L^{p_2(2-\theta)}(0,T;L^{q_2(2-\theta)}(\Omega))}\\
	\leq& C\varepsilon^{\theta\alpha }\| f\|^{\theta}_{L^{\theta p_{1}}(0,T;B_{\theta q_{1},c(\mathbb{N})}^{\alpha}(\Omega))}
	\|f\|^{2-\theta}_{L^{p_2(2-\theta)}(0,T;L^{q_2(2-\theta)}(\Omega))}\\
	\leq &o(\varepsilon^{\theta\alpha }),\ \text{as}\ \varepsilon\to 0,
	\end{aligned} $$
where  Lemma \ref{lem2.2} was used.

(2) The  H\"older inequality allows us to discover that
$$|I|\leq \f{1}{\varepsilon^{d}}\B[\int_{B_{\varepsilon}(0)}
\big|\eta \big(\f{y}{\varepsilon}\big)
 [|f(x-y)-f(x)|^{2-\theta}]\big|^{\f{q_{1}}{q_{1}-1}}dy
 \B]^{1-\f{1}{q_{1}}}\B[\int_{B_{\varepsilon}(0)}
|f(x-y)-f(x)|^{\theta q_{1}} dy\B]^{\f{1}{q_{1}}}.
$$
Using  the  H\"older inequality once again and Minkowski inequality, we observe that
$$\ba
&\|I\|_{L^{s}(\Omega)}\\ \leq&
C\B[\f{1}{\varepsilon^d}\int_{B_{\varepsilon}(0)}
\big
 \||f(x-y)-f(x)|^{2-\theta} \|_{L^{q_{2}}} ^{\f{q_{1}}{q_{1}-1}}dy\B]^{1-\f{1}{q_{1}}}\B[\int_{\Omega}\fbxoo
|f(x-y)-f(x)|^{\theta q_{1}} dydx\B]^{\f{1}{q_{1}}},
\ea$$
where we require  $\f{1}{s}=\f{1}{q_1}+\f{1}{q_2}$ and $q_{2}\geq \f{ q_{1}}{q_{1}-1}$.

We deduce from the H\"older inequality, $p_{2}\geq\f{q_{1}}{q_{1}-1}$  and  Minkowski inequality once again, as $\varepsilon\to0$,   that
\be\label{key1}\ba & \|I\|_{L^r(0,T;L^s((\Omega))}\\ \leq&C \B[\f{1}{\varepsilon^d}\int_{B_{\varepsilon}(0)}
 \| |f(x-y)-f(x)|^{2-\theta} \|_{L^{p_{2}}(0,T;L^{q_{2}})} ^{\f{q_{1}}{q_{1}-1}}dy\B]^{1-\f{1}{q_{1}}}\\&\times
 \B\| \int_{\Omega}\fbxoo
|f(x-y)-f(x)|^{\theta q_{1}} dydx\B\|_{L^{p_1}(0,T)}\\
  \leq& o(\varepsilon^{\theta\alpha }).
\ea\ee
Thanks to the  H\"older inequality and Remark \ref{rem2.1}, we know that,  as $\varepsilon\to 0$,
\be\label{key2}\begin{aligned}
\|II\|_{L^r(0,T;L^s(\Omega))}
	\leq  C\| |f-f^{\varepsilon}|^{ \theta}\|_{L^{p_1}(0,T;L^{q_1}(\Omega))}  \| |f-f^{\varepsilon}|^{2-\theta}\|_{L^{p_2}(0,T;L^{q_2}(\Omega))}
	\leq  o(\varepsilon^{\theta\alpha}).
	\end{aligned} \ee
Plugging  \eqref{key1} and \eqref{key2} into \eqref{CETI}, we arrive at the desired estimate. The proof of this lemma is completed.

	Then the proof of this lemma is completed.
\end{proof}

 \begin{lemma}\label{lem2.4}(\cite{[YWW],[WWY22]})
Let $\Omega$ denote the whole space $\mathbb{R}^{d}$ or the periodic domain $\mathbb{T}^{d}$.	Let $ p,q,p_1,q_1,p_2,q_2\in[1,+\infty)$ with
$\frac{1}{p}=\frac{1}{p_1}+\frac{1}{p_2},\frac{1}{q}=\frac{1}{q_1}+\frac{1}{q_2} $. Assume $f\in L^{p_1}(0,T;L^{q_1}(\Omega)) $ and $g\in
L^{p_2}(0,T;L^{q_2}(\Omega))$, then as $\varepsilon\to0$, it holds
	\begin{equation}\label{a4}
	\|(fg)^\varepsilon-f^\varepsilon
g^\varepsilon\|_{L^p(0,T;L^q(\Omega))}\rightarrow 0,
	\end{equation}
and
\begin{equation}\label{b7}
	\|(f\times g)^\varepsilon-f^\varepsilon\times g^\varepsilon\|_{L^p(0,T;L^q(\Omega))}\rightarrow 0.
\end{equation}
\end{lemma}
For the convenience of readers, we present the definition of the weak solutions of   the Euler   equation
 \eqref{Euler}.
\begin{definition}\label{eulerdefi}
We say that $(v,\Pi)$ is a weak solution of the incompressible homogeneous Euler equations \eqref{Euler} if  $v\in C_{\text{weak}}([0,T];L^{2}(\Omega))$, $\Pi \in L^{1}_{loc}([0,T]\times \Omega)$ with initial data $v_{0}\in L^{2}(\Omega)$ and
	\begin{enumerate}[(i)]
		\item  for every test vector field $\varphi\in C_{0}^{\infty}([0,T]\times\Omega )^{d}$, there holds,
		\begin{equation*}
			\begin{aligned}
			&\int_{\Omega}[v(x,T) \varphi(x,T)- v(x,0) \varphi(x,0)]\,dx \\
			=&\int_{0}^{T}\int_{\Omega}[v(x,t)\partial_{t}\varphi(x,t)+v(x,t)\otimes v(x,t)\cdot\nabla\varphi(x,t)+\Pi(x,t)\Div \varphi(x,t)]\,dxdt.
		\end{aligned}\end{equation*}
		\item[(ii)]
		$v$ is weakly divergence free, that is, for every test function $\psi\in C_{0}^{\infty}([0,T]\times\Omega)$,
		$$\int_0^T\int_{\Omega} v(x,t)\cdot\nabla\psi(x,t)\,dxdt=0.$$
		
	\end{enumerate}
\end{definition}
\section{Energy conservation of weak solutions for the ideal flow}
In this section, we consider the energy conservation
of weak solutions for the Euler equations
on the whole space $\mathbb{R}^d$ and   for the case of the torus $\mathbb{T}^d$, respectively.
\subsection{Energy conservation for ideal flows on $\mathbb{R}^d$  }

\begin{proof}[Proof of Theorem \ref{the1.1}] (1) One can achieve the proof of this part via
a modification the proof presented in \cite{[CCFS]} for $1<p\leq3$ (see also the proof of (1) in Theorem \ref{the1.3}  below). We provided  a  short proof here $1\leq p\leq3$. The interpolation inequality in Lebesgue space means that
\be\label{3.1} \|\Delta_{j}v\|_{L^{3}(\mathbb{R}^{d})}\leq \|\Delta_{j}v\|_{L^{2}(\mathbb{R}^{d})}^{1-\f{p}{3}}
 \|\Delta_{j}v\|_{L^{\f{2p}{p-1}}(\mathbb{R}^{d})}^{\f{p}{3}},\ \text{for\ any }\ p\in [1,3],\ee
which implies that
\be\label{3.2}\ba
 2^{\f13 j}\|\Delta_{j}v\|_{L^{3}(\mathbb{R}^{d})}\leq& \|\Delta_{j}v\|_{L^{2}(\mathbb{R}^{d})}^{1-\f{p}{3}}
\B[2^{j\f1p}\|\Delta_{j}v\|_{L^{\f{2p}{p-1}}(\mathbb{R}^{d})}\B]^{\f{p}{3}}\\
\leq& C\| v\|_{L^{2}(\mathbb{R}^{d})}^{1-\f{p}{3}}  \|v\|^{\f{p}{3}}_{
B^{\f1p}_{\f{2p}{p-1},\infty}(\mathbb{R}^{d})}.
\ea\ee
As a consequence, we discover that
$$ \|v\|_{
B^{\f13}_{3,\infty}(\mathbb{R}^{d})}\leq  C\| v\|_{L^{2}(\mathbb{R}^{d})}^{1-\f{p}{3}}  \|v\|^{\f{p}{3}}_{
B^{\f1p}_{\f{2p}{p-1},\infty}(\mathbb{R}^{d})},$$
which turns out that
$$ \|v\|_{L^{3}(0,T;
B^{\f13}_{3,\infty}(\mathbb{R}^{d}))}\leq  C\| v\|_{L^{\infty}(0,T;L^{2}(\mathbb{R}^{d}))}^{1-\f{p}{3}}  \|v\|^{\f{p}{3}}_{L^{p}(0,T;
B^{\f1p}_{\f{2p}{p-1},\infty}(\mathbb{R}^{d}))}.$$
Moreover, it follows from \eqref{3.2} that
$$\ba 2^{\f13 j}\|\Delta_{j}v\|_{L^{3}(\mathbb{R}^{d})}\leq& C\| v\|_{L^{2}(\mathbb{R}^{d})}^{1-\f{p}{3}}
\B[2^{j\f1p}\|\Delta_{j}v\|_{L^{\f{2p}{p-1}}(\mathbb{R}^{d})}\B]^{\f{p}{3}},
\ea$$
where $C$ is independent of $j$.
This  ensures that $v\in L^{p}(0,T;
B^{\f1p}_{\f{2p}{p-1},c(\mathbb{N})}(\mathbb{R}^{d}))$ means $v\in L^{3}(0,T;
B^{\f13}_{3,c(\mathbb{N})}(\mathbb{R}^{d}))$. Thus, the proof of this part is completed.

(2)        We   multiply the Euler equations \eqref{Euler}   by $S_{N}(S_{N}v)$ and utilize
 integration by parts to get
$$
\f{1}{2}\f{d}{dt}\int_{\mathbb{R}^{2}} |S_{N}v |^{2}dx=- \int_{\mathbb{R}^3}
\partial_{j}(v_{j}v_{i})S^{2}_{N}v_{i}\,dx.
$$
Performing a time integration, we know that
\be\label{zheng1}
\f12\|S_{N}v\|_{L^{2}(\mathbb{R}^{3})}^{2}-\f12\|S_{N}v_{0}\|_{L^{2}(\mathbb{R}^{3})}^{2} =-\int_{0}^{t}\int_{\mathbb{R}^3}
\partial_{j}(v_{j}v_{i})S^{2}_{N}v_{i}\,dxds.
\ee
Our plan is to show that, as $N\rightarrow+\infty$, the right hand side of the above equation \eqref{zheng1} converges  to 0.

To do this, adapting the incompressibility condition $\operatorname{div}v=0$,   we
discover that
$$\ba
-\int_{0}^{t}\int_{\mathbb{R}^3}
\partial_{j}(v_{j}S^{2}_{N}v_{i})S^{2}_{N}v_{i}\,dxds
=&-\f12\int_{0}^{t}\int_{\mathbb{R}^3}
v_{j}\partial_{j}(S^{2}_{N}v_{i})^{2} \,dxds=0.
\ea$$
As  a consequence, we rewrite the right hand side of \eqref{zheng1}
as
\begin{equation*}
\begin{split}
 -\int_{0}^{t}\int_{\mathbb{R}^3}
\partial_{j}(v_{j}v_{i})S^{2}_{N}v_{i}\,dxds
=  \int_{0}^{t}\int_{\mathbb{R}^3}
S_{N}\big[v_{j}(\text{I}_{\textnormal d}-S^{2}_{N})v_{i}\big]S_{N}\partial_{j}v_{i}\,dxds.
\end{split}
\end{equation*}
Utilizing  the H\"older inequality, we find
\be\ba\label{3.4}
&\left|\int_{0}^{t}\int_{\mathbb{R}^3}S_{N}\big[v_{j}
(I-S^{2}_{N})v_{i}\big]S_{N}\partial_{j}v_{i}\,dxds\right| \\
\leq&C\big\| S_{N}\big[v_{j}(\text{I}_{\textnormal d}-S^{2}_{N})v_{i}\big]\big\|_{L^{\f{p}{p-1}}(0,T;L^{\f{2dp}{dp+d-2p+2}}
(\mathbb{R}^{3}))}  \big\| \nabla S_{N}v \big\|_{L^{p}(0,T;L^{\f{2dp}{(d+2)(p-1)}}(\mathbb{R}^{3}))}\\
\leq&  C\big\|  v_{j}(\text{I}_{\textnormal d}-S^{2}_{N})v_{i}\big\|_{L^{\f{p}{p-1}}(0,T;L^{\f{2dp}{dp+d-2p+2}}
(\mathbb{R}^{3}))}  \big\| S_{N}\partial_{j}v \big\|_{L^{p}(0,T;L^{\f{2dp}{(d+2)(p-1)}}(\mathbb{R}^{3}))}\\
\leq& C\|  v   \|_{L^{\f{2p}{p-1}}(0,T;L^{L^{\f{4dp}{dp+d-2p+2}}}(\mathbb{R}^{3}))} \big\|\big(\text{I}_{\textnormal d}-S^2_{N}\big)v \big\|_{L^{\f{2p}{p-1}}(0,T;L^{\f{4dp}{dp+d-2p+2}}(\mathbb{R}^{3}))}  \| \nabla  v \|_{L^{p}(0,T;L^{\f{2dp}{(d+2)(p-1)}}(\mathbb{R}^{3}))}.
\ea\ee
We apply the  Gagliardo-Nirenberg inequality to get
$$
\|  v   \|_{L^{\f{4dp}{dp+d-2p+2}}(\mathbb{R}^{3}) }\leq  C\|  v   \|^{\f{6-2p-pd+3d}{2d+4}}_{L^{2}(\mathbb{R}^{3}) }
 \|  \nabla v   \|^{\f{ pd+2p-2-d}{ 2d+4}}_{L^{\f{2dp}{(d+2)(p-1)}}(\mathbb{R}^{3}) },$$
 from which follows that
\be\label{sob2}
\|  v   \|_{L^{\f{2p}{p-1}}(0,T;L^{\f{4dp}{dp+d-2p+2}}(\mathbb{R}^{3}))}\leq  C\|  v   \|^{\f{6-2p-pd+3d}{2d+4}}_{L^{\infty}(0,T;L^{2}(\mathbb{R}^{3})) }\| \nabla  v \|_{L^{p}(0,T;L^{\f{2dp}{(d+2)(p-1)}}(\mathbb{R}^{3}))}^{\f{ pd+2p-2-d}{ 2d+4}}.\ee
Invoking  Littlewood-Paley theory  and Dominated Convergence Theorem, we see that, as $N\rightarrow\infty$,
$$\big\|\big(\text{I}_{\textnormal d}-S^2_{N}\big)v \big\|_{L^{\f{2p}{p-1}}(0,T;L^{\f{4dp}{dp+d-2p+2}}(\mathbb{R}^{3}))}
 \rightarrow0,
$$
which follows from that, as $N\rightarrow\infty$,
$$
-\int_{0}^{t}\int_{\mathbb{R}^3}
\partial_{j}(v_{j}v_{i})S^{2}_{N}v_{i}\,dxds\rightarrow0.
$$
Taking $N\rightarrow\infty$ in    \eqref{zheng1}, we derive from \eqref{zheng1} that
$$ \f12\|v(t)\|_{L^{2}(\mathbb{R}^{3})}^{2}= \f12\|v_0\|_{L^{2}(\mathbb{R}^{3})}^{2}.$$
Thus we  complete the proof.

\end{proof}

\subsection{Energy conservation for invicid  flows on torus }
Unlike the interpolation and Littlewood-Paley theory utilized in the last sub-section, the mollifier kernel and Constantin-E-Titi type estimates are applied to deal with  periodic case.

\begin{proof}[Proof of Theorem \ref{the1.2}](1)-(2) Let us begin by
mollifying the equations \eqref{Euler}  in spatial direction to get
\be\ba\label{rmhd}
&\partial_{t}{v^{\varepsilon}} +  \text{div}(v\otimes v)^{\varepsilon} +\nabla\Pi^{\varepsilon}= 0.
\ea\ee
The standard energy estimate means that
\begin{equation}\ba\label{3.6}
\f12\|v^{\varepsilon}(T)\|^{2}_{L^{2}(\mathbb{T}^d)}
-\f12\|v^{\varepsilon}(0)\|^{2}_{L^{2}(\mathbb{T}^d)} =-\int_{0}^{T}\int_{\mathbb{T}^d}
\text{div}(v\otimes v)^{\varepsilon}v^{\varepsilon}dxdt.
\ea\end{equation}
In view  of divergence-free condition and the integration by parts,  we reformulate  the right hand side of  \eqref{3.6}  as
\begin{equation}\label{c22}\ba
&-\int_{0}^{T}\int_{\mathbb{T}^d}
\text{div}(v\otimes v)^{\varepsilon}v^{\varepsilon}dxdt\\
=&-\int_{0}^{T}\int_{\mathbb{T}^d}
\text{div}\big[(v\otimes v)^{\varepsilon}-
(v^{\varepsilon}\otimes v^{\varepsilon})\big]v^{\varepsilon}dxdt-
\int_{0}^{T}\int_{\mathbb{T}^d}\text{div}(v^{\varepsilon}\otimes v^{\varepsilon})v^{\varepsilon}dxdt\\
=&-\int_{0}^{T}\int_{\mathbb{T}^d}
\text{div}\big[(v\otimes v)^{\varepsilon}-
(v^{\varepsilon}\otimes v^{\varepsilon})\big]v^{\varepsilon}dxdt\\
=&\int_{0}^{T}\int_{\mathbb{T}^d}
\big[(v\otimes v)^{\varepsilon}-
(v^{\varepsilon}\otimes v^{\varepsilon})\big]:\nabla v^{\varepsilon}dxdt.
\ea\end{equation}
Plugging this  into \eqref{3.6}, we arrive at
\begin{equation}\label{c27}
	\begin{aligned}
	 \f12\|v^{\varepsilon}(T)\|^{2}_{L^{2}(\mathbb{T}^d)}
	 	-\f12\|v^{\varepsilon}(0)\|^{2}_{L^{2}(\mathbb{T}^d)}
	 	=\int_{0}^{T}\int_{\mathbb{T}^d}
	\big[(v\otimes v)^{\varepsilon}-
	(v^{\varepsilon}\otimes v^{\varepsilon})\big]:\nabla v^{\varepsilon}dxdt.
	\end{aligned}
	\end{equation}
On the one hand, in virtue of the H\"older inequality, we have
$$\ba
&\B|\int_{0}^{T}\int_{\mathbb{T}^d}
	\big[(v\otimes v)^{\varepsilon}-
	(v^{\varepsilon}\otimes v^{\varepsilon})\big]:\nabla v^{\varepsilon}dxdt\B|\\
\leq& C \|(v^{\varepsilon}\otimes v^{\varepsilon}) -(v\otimes v)^{\varepsilon}\|_{L^{\f{p}{p-1}}(0,T;L^{\f{q}{q-1}}(\mathbb{T}^d))}
 \|\nabla v^{\varepsilon}\|_{L^{p}(0,T;L^{q}(\mathbb{T}^d))}.
\ea$$
We select
   $p_{2}(2-\theta)=\infty,q_{2}(2-\theta)=2$, $s=\f{q}{q-1}$ and $r=\f{p}{p-1}$ in Lemma  \ref{lem2.3} to derive that  $\theta p_{1}=\f{p\theta}{p-1}$, $\theta q_{1}=\f{2 q\theta}{q\theta-2}$
and let $\theta p_{1}= p$ and $\theta q_{1}=q$, we can obtain
\begin{equation*}
	\begin{aligned}
\|(v^{\varepsilon}\otimes v^{\varepsilon}) -(v\otimes v)^{\varepsilon}\|_{L^{\f{p}{p-1}}(0,T;L^{\f{q}{q-1}}(\mathbb{T}^d))}
\leq o(\varepsilon^{\theta\alpha})\leq o(\varepsilon^{(p-1)\alpha}) ,		
	\end{aligned}
\end{equation*}
where we have used the facts that  $\theta =\frac{2}{q-2}=p-1$ and $q=\frac{2p}{p-1}$.

On the other hand, in light of Lemma \eqref{lem2.2}, we have
$$\|\nabla v^{\varepsilon}\|_{L^{p}(0,T;L^{q}(\mathbb{T}^d))}
 \leq o(\varepsilon^{ \alpha-1}).$$
Hence,
 $$\ba
&\B|\int_{0}^{T}\int_{\mathbb{T}^d}
	\big[(v\otimes v)^{\varepsilon}-
	(v^{\varepsilon}\otimes v^{\varepsilon})\big]:\nabla v^{\varepsilon}dxdt\B|
\leq o(\varepsilon^{ (p-1)\alpha+\alpha-1})\leq  o(\varepsilon^{ p\alpha-1}).
\ea$$
Then letting $\alpha =\frac{1}{p},q=\frac{2p}{p-1}$ and passing to the limit of $\varepsilon$ in \eqref{c27}, we get the energy conservation.

 (3) We deduce from the H\"older inequality that
 $$\ba
&\B|\int_{0}^{T}\int_{\mathbb{T}^d}
	\big[(v\otimes v)^{\varepsilon}-
	(v^{\varepsilon}\otimes v^{\varepsilon})\big]:\nabla v^{\varepsilon}dxdt\B|\\
\leq& C \|(v^{\varepsilon}\otimes v^{\varepsilon}) -(v\otimes v)^{\varepsilon}\|_{L^{\f{p}{p-1}}(0,T;L^{\f{2dp}{dp+d-2p+2}}(\mathbb{T}^d))}
 \|\nabla v^{\varepsilon}\|_{L^{p}(0,T;L^{\f{2dp}{(d+2)(p-1)}}(\mathbb{T}^d))}
\ea$$
From an analogue of inequality \eqref{sob2} on bounded domain  and $v\in   L^{\infty}(0,T;L^{2}(\mathbb{T}^{d}))$, we derive from $\nabla v\in L^{p}(0,T;L^{\f{2dp}{(d+2)(p-1)}}(\mathbb{R}^{d})) $ that $v\in L^{\f{2p}{p-1}}(0,T;L^{\f{4dp}{dp+d-2p+2}}(\mathbb{T}^{d}))$. Employing Lemma \ref{lem2.4}, we find, as $\varepsilon\rightarrow0,$
$$\int_{0}^{T}\int_{\mathbb{T}^d}
	\big[(v\otimes v)^{\varepsilon}-
	(v^{\varepsilon}\otimes v^{\varepsilon})\big]:\nabla v^{\varepsilon}dxdt\rightarrow0.$$
The proof of this theorem is completed.
\end{proof}

\section{Helicity conservation  of weak solutions for the ideal flow}
In this section, we are devoted to the proof of Theorem \ref{the1.3} and \ref{the1.4}, involving the minimum regularity of the weak solutions to guarantee the helicity conservation for the incompressible Euler equations.
\subsection{Helicity conservation for ideal flows on $\mathbb{R}^3$ }
\begin{proof}[Proof of Theorem \ref{the1.3}]
(1)
For any $p\in [1,3]$, by virtue of
$$2^{\f23 j}\|\Delta_{j}u\|_{L^{3}(\mathbb{R}^{3})}\leq C\|\Delta_{j}u\|_{L^{2}(\mathbb{R}^{3})}^{1-\f{p}{3}}
\B[2^{j\f2p}\|\Delta_{j}u\|_{L^{\f{2p}{p-1}}(\mathbb{R}^{3})}\B]^{\f{p}{3}} $$
and the helicity conservation result in \cite{[CCFS]},
one can argue almost exactly as the proof in the Theorem \ref{the1.1} to complete the
proof this part. We omit the detail here.  As compensation, we present the proof via the Littlewood-Paley theory for $2<p\leq3$.

Combining the vorticity equations \eqref{vorticityeq} and some straightforward calculations,
we conclude by integration by parts that
\be\label{4.1}\ba
&\f{d}{dt}\int_{\mathbb{R}^3}(S_{N}v \cdot S_{N}\omega )dx\\
=&\int_{\mathbb{R}^3} S_{N}v_{t} \cdot
S_{N}\omega +S_{N}\omega_{t} \cdot S_{N}v dx\\
=&\int_{\mathbb{R}^3}-\left[S_{N}(\omega\times v)+\nabla S_{N}(\Pi+\f12|v|^{2}) \right]\cdot S_{N}\omega +
 \left(\text{curl}S_{N}\left( v\times\omega\right)\right) \cdot S_{N}v dx\\
=&\int_{\mathbb{R}^3}-S_{N}\left(\omega\times v\right) \cdot S_{N}\omega +
S_{N}\left(   v\times\omega \right) \cdot S_{N}\omega dx.
 \ea\ee
 Using the fact $(   S_{N}v \times S_{N}\omega )\cdot S_{N}\omega =   S_{N}v \cdot(S_{N}\omega  \times S_{N}\omega )=0$, we reformulate \eqref{4.1} as
 \be\label{4.2}\ba
\f{d}{dt}\int_{\mathbb{R}^3}(S_{N}v \cdot S_{N}\omega )dx
 =&\int_{\mathbb{R}^3}-2\Big[S_{N}(\omega\times v)  -S_{N}\omega \times S_{N}v  \Big]\cdot S_{N}\omega dx.
 \ea\ee
 To control the right hand side of \eqref{4.2}, we recall the identity
$$\nabla(A\cdot B)=A\cdot\nabla B+B\cdot\nabla A+A\times\text{curl}B+B\times\text{curl}A$$
which together with
$ \Div \omega =0$ yields that
$$ \text{div}(v\otimes v) = \f{1}{2}\nabla |v|^{2}+\omega\times v,$$
which gives
$$ \omega\times v = \text{div}(v\otimes v)-\f{1}{2}\nabla |v|^{2}.$$
As a consequence, we deduce that
$$S_{N}(\omega\times v )=S_{N}(v\cdot\nabla v)-\f{1}{2}\nabla S_{N}(|v|^{2}) = \text{div}S_{N}(v\otimes v) -\f{1}{2}\nabla S_{N}(|v|^{2})$$
and
$$ S_{N}\omega \times S_{N}v  =S_{N}v \cdot\nabla S_{N}v -\f{1}{2}\nabla |S_{N}v |^{2}= \text{div} (S_{N}v  \otimes S_{N}v )-\f{1}{2}\nabla |S_{N}v |^{2}.$$
Plugging the latter two equations into \eqref{4.2}, we use integration by parts to obtain that
\be \ba\label{4.3}
&\f{d}{dt}\int_{\Omega} S_{N}v \cdot S_{N}\omega dx\\=&-2\int_{\Omega}\text{div}\big(S_{N}(v\otimes v)  -(S_{N}v \otimes S_{N}v)\big)\cdot S_{N}\omega +\frac{1}{2}\nabla \left(|S_{N}v |^2 -S_{N}(|v|^2)\right)\cdot S_{N}\omega dx\\
=& 2\int_{\Omega} \big(S_{N}(v\otimes v)-S_{N}v \otimes S_{N}v \big): \nabla S_{N}\omega dx.
    \ea\ee
    Making use of the H\"older inequality, we have
$$
	\begin{aligned}
	&\B|\int_{\Omega} \big(S_{N}(v\otimes v)-S_{N}v \otimes S_{N}v \big): \nabla S_{N}\omega dx \B|\\
\leq	&C  \|S_{N}(v\otimes v)-S_{N}v \otimes S_{N}v \|_{ L^{\f{q}{q-1}} (\mathbb{R}^3)}  \|\nabla S_{N}\omega \|_{ L^{q}  (\mathbb{R}^3)}.
\end{aligned}$$
Thanks to the Constantin-E-Titi identity in \cite{[CET]}, we observe that
$$\ba
&S_{N}(v\otimes v) -S_{N}v \otimes S_{N}v \\
=&  2^{ 3N}\int_{\mathbb{R}^3}\tilde{h}(2^{N}y)[v_{i}(x-y)-v_{i}(x)][v_{j}(x-y)-v_{j}(x)]dy-(v_{i}-S_{N}v_{i})(v_{j}-S_{N}v_{j}),
\ea$$
where we used  $2^{ 3N}\int_{\mathbb{R}^3}\tilde{h}(2^{N}y)dy=\mathcal{F}(\tilde{h}(\cdot))|_{\xi=0}=1.$

 For $0\leq\theta\leq2$, we deduce from   the Minkowski inequality and $v\in L^{\infty}(0,T;L^{2} (\mathbb{R}^3))$ that,
$$\ba
 &\|S_{N}(v_{i}v_{j})  -S_{N}v_{i}S_{N}v_{j}\|_{  L^{\f{q}{q-1}}  (\mathbb{R}^3)}
 \\\leq& 2^{3N}\int_{\mathbb{R}^3}|\tilde{h}(2^{N}y)|\| v (x-y)-v (x) \|^{\theta}_{ L^{\f{2q\theta}{q\theta-2}}  (\mathbb{R}^3)}
 \| v(x-y)-v\|^{2-\theta}_{ L^{2}(\mathbb{R}^3) }dy
\\&+\| v_{i}-S_{N}v_{i} \|^{\theta}_{ L^{\f{2q\theta}{q\theta-2}}  (\mathbb{R}^3)} \| v_j-S_{N}v_j \|^{2-\theta}_{ L^{2}(\mathbb{R}^3) } \\  \leq& 2^{3N}\int_{\mathbb{R}^3}|\tilde{h}(2^{N}y)|\| v (x-y)-v (x) \|^{\theta}_{ L^{\f{2q\theta}{q\theta-2}}  (\mathbb{R}^3)}
  dy
 +\| v_i -S_{N}v_i  \|^{\theta}_{ L^{\f{2q\theta}{q\theta-2}}  (\mathbb{R}^3)} \\
=&I+II.
\ea$$
On the one hand, applying the Bernstein inequality  in Lemma \ref{berinequ}, we observe that
\be\ba\label{4.4}
 \|v (x-y)-v (x)\|_{ L^{\f{2q\theta}{q\theta-2}}  (\mathbb{R}^3)}
\leq C\B[\sum_{j< N}2^{j}|y|\| {\Delta}_{j}v\|_{ L^{\f{2q\theta}{q\theta-2}}  (\mathbb{R}^3)}+\sum_{j\geq N}\| {\Delta}_{j}v\|_{ L^{\f{2q\theta}{q\theta-2}}  (\mathbb{R}^3)}\B].
\ea\ee
Before going further, we write
$$\ba
\Gamma(j)=\left\{\begin{aligned}
	&2^{j\alpha},~~~~~~~~\text{if}~j\leq0,\\
	& 2^{-(1-\alpha)j},~~\text{if}~j>0,
\end{aligned}\right.
\ea$$
and  $d_j=2^{j\alpha}\| {\Delta}_{j}v\|_{L^{\f{2q\theta}{q\theta-2}} (\mathbb{R}^{3})}$.

Then we can  rewrite \eqref{4.4} as
$$\ba
&\| v(x-y)-v(x) \|_{ L^{\f{2q\theta}{q\theta-2}}  (\mathbb{R}^3)}\\
\leq&C\B( 2^{N(1-\alpha)}|y|\sum_{j< N}2^{-(N-j)(1-\alpha)}2^{j\alpha}\| {\Delta}_{j}v\|_{ L^{\f{2q\theta}{q\theta-2}}  (\mathbb{R}^3)}+2^{-\alpha N}\sum_{j\geq N} 2^{(N-j)\alpha} 2^{j\alpha}\| {\Delta}_{j}v\|_{ L^{\f{2q\theta}{q\theta-2}}  (\mathbb{R}^3)}\B)
\\
\leq& C(2^{N}|y| +1)2^{ -\alpha N}\left(\Gamma\ast {d}_j\right)(N),
\ea$$
which turns out that
$$\ba
 \| v(x-y)-v(x) \|^{\theta}_{ L^{\f{2q\theta}{q\theta-2}}  (\mathbb{R}^3)}
\leq& C(2^{N}|y| +1)^{\theta}2^{ -\alpha\theta N}\left(\Gamma\ast {d}_j\right)^{\theta}(N).
\ea$$
Some straightforward calculations yield that
$$ \sup_{N} 2^{3N}\int_{\mathbb{R}^3}  |\tilde{h}(2^{N}y)|(2^{N}|y| +1)^{\theta}dy<\infty,$$
which means that
$$I\leq C 2^{ -\alpha\theta N}\left(\Gamma \ast {d}_j\right)^{\theta}(N).$$
On the other hand, notice that
$$|II| \leq\| v -{S}_{N}v  \|^{\theta}_{ L^{\f{2q\theta}{q\theta-2}}  (\mathbb{R}^3)}  \leq \B(\sum_{j\geq N}\| {\Delta}_{j}v\|_{ L^{\f{2q\theta}{q\theta-2}}  (\mathbb{R}^3)}\B)^{\theta}
\leq C2^{ -\alpha\theta  N}\left(\Gamma\ast   {d}_j \right)^{\theta}(N),
$$
where  we used $N>0.$

Therefore, we end up with,
\be\ba\label{4.5}\|S_{N}(v_{i}v_{j})  -S_{N}v_{i}S_{N}v_{j}\|_{  L^{\f{q}{q-1}}  (\mathbb{R}^3)}\leq  C2^{ -\theta\alpha  N}\left(\Gamma\ast   {d}_j \right)^{\theta}(N),
\ea\ee
Owing to the Bernstein inequality, we conclude that
  \be\label{4.6}
\|\nabla S_{N}\omega\|_{ L^{q} (\mathbb{R}^3) }\leq \sum_{j\leq N}2^{2j} \| {\Delta}_{j}v\|_{L^{q} (\mathbb{R}^3)}\leq  2^{N(2 -\alpha)}\left(\Gamma\ast {d}_j\right)(N),
\ee
Combining \eqref{4.5} and \eqref{4.6}, we infer that
 \be\ba\label{4.7}
&\B| \int_{\mathbb{R}^3} \big(S_{N}(v\otimes v)^{\varepsilon}-S_{N}v \otimes S_{N}v \big): \nabla S_{N}\omega dx \B|\\
\leq&C2^{(2- \alpha-\theta\alpha )N}\left(\Gamma\ast  {d}_j\right)^{\theta}(N) \left( \Gamma\ast {d}_j\right)(N).
\ea\ee
 To proceed further, we pick $\alpha$ such that $0<\alpha<1$ and $2- \alpha-\theta\alpha=0$. Hence, we have $\Gamma \in  l^{1}(\mathbb{Z})$.
Letting $\alpha=\frac{2}{p}$ and $q=\frac{2p}{p-1}$, it follows from \eqref{4.7}
that
 $$\ba
&\int_{0}^{T}  \B|\int_{\mathbb{R}^3} \big(S_{N}(v\otimes v)-S_{N}v \otimes S_{N}v \big): \nabla S_{N}\omega dx \B|dt\\
 \leq &C\int_0^T\B|\left[\Gamma\ast  {d}_j\right]^{\theta}(N)\left[ \Gamma\ast  {d}_j\right](N)\B| dt
\\  \leq  &C\|v\|^{\theta}_{L^{\f{p\theta}{p-1}}(0,T;B^{\alpha}_{\f{2q\theta}{q\theta-2},\infty}(\mathbb{R}^{3}))} \|v\|_{L^{p}(0,T;B^{\alpha}_{q,\infty}(\mathbb{R}^{3}))} \\  \leq  &C \|v\|^{\theta+1}_{L^{p}(0,T;B^{\alpha}_{q,\infty}(\mathbb{R}^{3}))} <\infty,
\ea$$
where we used the facts that
 $ \f{p\theta}{p-1}=p$ and $\f{2q\theta}{q\theta-2}=q.$

Invoking the
the dominated convergence theorem, we get
\be\ba\label{4.8}
&\int_{0}^{T}\B| \int_{\mathbb{R}^{3}} \big(S_{N}(v\otimes v)-S_{N}v \otimes S_{N}v \big): \nabla S_{N}\omega dx  \B|dt\\ \leq &C\int_0^T\B|\left[\Gamma\ast  {d}_j\right]^{\theta}(N)\left[ \Gamma\ast  {d}_j\right](N)\B| dt \rightarrow 0, ~\text{as}\  N\rightarrow +\infty.
\ea\ee
By integrating \eqref{4.3} in time over $(0,T)$, we find
\be\label{4.9}\ba
&\int_{\mathbb{R}^{3}}S_{N} v(x,T)\cdot S_{N}\omega(x,T)dx- \int_{\mathbb{R}^{3}} S_{N}v(x,0)\cdot S_{N}\omega(x,0)dx
\\=& 2 \int_{0}^{T}\int_{\mathbb{R}^{3}} \big(S_{N}v\otimes S_{N}v-S_{N}(v\otimes v)\big):\nabla S_{N}\omega dxdt.
\ea\ee
With \eqref{4.8}  in hand, taking $N\rightarrow+\infty $ in \eqref{4.9}, we finish the proof of this part.

(2) We    deduce from the Gagliardo-Nirenberg inequality that
\be\label{4.10}
\|v\|_{L^{\f{ p}{ p-2}}(0,T;L^{\f{3p}{7-2p}}(\mathbb{R}^{3}))}\leq
C\|v\|^{\f{21-7p}{7}}_{L^{\infty}(0,T;L^{2}(\mathbb{R}^{3}))}
\|\nabla v\|^{\f{ 7p-14}{7}}_{ L^{p}(0,T;L^{\f{6p}{5p-7}}(\mathbb{R}^{3}))}.\ee
In addition, there holds
$$\ba
&\B|\int_{0}^{T}\int_{\Omega} \Big(S_{N}(\omega\times v)  -S_{N}\omega \times S_{N}v  \Big)\cdot S_{N}\omega dxdt\B|\\
\leq  & \|S_{N}(\omega\times v)  -S_{N}\omega \times S_{N}v \|_{ L^{\f{p}{p-1}}(0,T;L^{\f{6p}{p+7}}(\mathbb{R}^{3}))}   \|\nabla v\|_{ L^{p}(0,T;L^{\f{6p}{5p-7}}(\mathbb{R}^{3}))}.
\ea$$
It is enough to apply Lemma \ref{lem2.4} to finish the proof.

(3)-(4)   Before going further, we set
$$
\Gamma_{2}(j)=\left\{\begin{aligned}
	&2^{j\beta},~~~~~~~~\text{if}~j\leq0,\\
	& 2^{-(1-\beta)j},~~\text{if}~j>0,
\end{aligned}\right.
$$
and $ {\tilde{d}}_j= 2^{j\beta}\| {\Delta}_{j}\omega\|_{L^{\f{2q\theta}{q\theta-2}} (\mathbb{R}^{3})}$.

We mimic the derivation of \eqref{4.6} to get
\be\label{4.11}
\|\nabla S_{N}\omega\|_{ L^{q} (\mathbb{R}^3) }\leq \sum_{j\leq N}2^{ j} \| {\Delta}_{j}\omega\|_{L^{q} (\mathbb{R}^3)}\leq  2^{N(1 -\beta)}\left(\Gamma_{2}\ast {\tilde{d}}_j\right)(N),
\ee
A combination of \eqref{4.5} and \eqref{4.11}, we arrive at
\be\ba\label{4.12}
&\B| \int_{\mathbb{R}^3} \big(S_{N}(v\otimes v)^{\varepsilon}-S_{N}v \otimes S_{N}v \big): \nabla S_{N}\omega dx \B|\\
\leq&C2^{(1- \beta-\theta\alpha )N}\left(\Gamma\ast  {d}_j\right)^{\theta}(N) \left(\Gamma_{2}\ast {\tilde{d}}_j\right)(N).
\ea\ee
We conclude by the H\"older inequality that
\be\ba\label{4.13}
&\int_{0}^{T}\B| \int_{\mathbb{R}^3} \big(S_{N}(v\otimes v)^{\varepsilon}-S_{N}v \otimes S_{N}v \big): \nabla S_{N}\omega dx  \B|dt\\ \leq &C\int_0^T\left(\Gamma \ast  {d}_j\right)^{\theta}(N)\left( \Gamma_{2}\ast \tilde{d}_j\right)(N) dt
\\  \leq  &C\|v\|^{\theta}_{L^{\f{p\theta}{p-1}}(0,T;B^{\alpha}_{\f{2q\theta}{q\theta-2},\infty})} \|\omega\|_{L^{p}(0,T;B^{\beta}_{q,\infty})} <\infty,
\ea\ee
where we require
$$ \f{p\theta}{p-1}=k, \f{2q\theta}{q\theta-2}=\ell, 1- \beta-\theta\alpha=0.$$
By means of
the dominated convergence theorem, we infer that,  as $ N\rightarrow +\infty$,
$$\ba
&\int_{0}^{T}\B| \int_{\Omega} \big(S_{N}(v\otimes v)^{\varepsilon}-S_{N}v \otimes S_{N}v \big): \nabla S_{N}\omega dx  \B|dt\\ \leq &C\int_0^T\left(\Gamma \ast  {d}_j\right)^{\theta}(N)\left( \Gamma_{2}\ast \tilde{d}_j\right)(N) dt\rightarrow 0.
\ea$$
The proof of this part is completed.
\end{proof}

\subsection{Helicity conservation for ideal flows on $\mathbb{T}^3$ }

\begin{proof}[Proof of Theorem \ref{the1.4}]
(1)-(2)
Arguing in the same manner as in the derivation of \eqref{4.3}, we find
$$\ba
\f{d}{dt}\int_{\mathbb{T}^{3}} v^{\varepsilon}\cdot\omega^{\varepsilon}dx
=& 2\int_{\mathbb{T}^{3}} \big[(v\otimes v)^{\varepsilon}-v^{\varepsilon}\otimes v^{\varepsilon}\big] : \nabla\omega^{\varepsilon}dx,
    \ea$$
    which together with integration with respect to  time leads to
    \be\label{4.14}\ba
&\int_{\mathbb{T}^{3}} v^{\varepsilon}(x,T)\cdot\omega^\varepsilon(x,T)dx- \int_{\mathbb{T}^{3}} v^{\varepsilon}(x,0)\cdot\omega^\varepsilon(x,0)dx
\\=& 2 \int_{0}^{T}\int_{\mathbb{T}^{3}} \big(v^{\varepsilon}\otimes v^{\varepsilon}-(v\otimes v)^{\varepsilon}\big):\nabla\omega^{\varepsilon} dxdt.
\ea\ee
Making full use of the H\"older inequality, we discover that
   \be\ba\label{4.15}
&\B|\int_{0}^{T}\int_{\mathbb{T}^{3}}
	\big((v\otimes v)^{\varepsilon}-
	(v^{\varepsilon}\otimes v^{\varepsilon})\big):\nabla \omega^{\varepsilon}dxdt\B|\\
\leq& C \|(v^{\varepsilon}\otimes v^{\varepsilon}) -(v\otimes v)^{\varepsilon}\|_{L^{\f{p}{p-1}}(0,T;L^{\f{q}{q-1}}(\mathbb{T}^{3}))}
 \|\nabla \omega^{\varepsilon}\|_{L^{p}(0,T;L^{q}(\mathbb{T}^{3}))}.
\ea\ee
Picking  $p_{2}(2-\theta)=\infty,q_{2}(2-\theta)=2$, $s=\f{q}{q-1}$ and $r=\f{p}{p-1}$ in Lemma  \ref{lem2.3} to derive that  $\theta p_{1}=\f{p\theta}{p-1}$, $\theta q_{1}=\f{2 q\theta}{q\theta-2}$
and letting $\theta p_{1}= p$ and $\theta q_{1}=q$, we can obtain
$$ \|(v^{\varepsilon}\otimes v^{\varepsilon}) -(v\otimes v)^{\varepsilon}\|_{L^{\f{p}{p-1}}(0,T;L^{\f{q}{q-1}}(\mathbb{T}^{3}))}
 \leq o(\varepsilon^{\theta\alpha}),$$
where we have used the facts that  $\theta =\frac{2}{q-2}=p-1$ and $q=\frac{2p}{p-1}$.

By virtue of Lemma \ref{lem2.2}, we know that
$$\|\nabla \omega^{\varepsilon}\|_{L^{p}(0,T;L^{q}(\mathbb{T}^{3}))}
 \leq o(\varepsilon^{ \alpha-2}),\ \text{as}\ \varepsilon\to 0.$$
Putting the above estimate together, we observe that
 $$\ba
&\B|\int_{0}^{T}\int_{\mathbb{T}^{3}}
	\big((v\otimes v)^{\varepsilon}-
	(v^{\varepsilon}\otimes v^{\varepsilon})\big):\nabla v^{\varepsilon}dxdt\B|
\leq o(\varepsilon^{ \theta\alpha+\alpha-2}).
\ea$$
Letting $\varepsilon\rightarrow0$ in \eqref{4.14}, we obtain the energy conservation.

 (3) It follows from the H\"older inequality that
 \begin{equation}\label{4.16}\ba
&\B|\int_{0}^{T}\int_{\mathbb{T}^{3}}
	\big((v\times \omega)^{\varepsilon}-
	(v^{\varepsilon}\times \omega^{\varepsilon})\big): \omega^{\varepsilon}dxdt\B|\\\leq  & C\|(v\times \omega)^{\varepsilon}-
	(v^{\varepsilon}\times \omega^{\varepsilon}) \|_{ L^{\f{p}{p-1}}(0,T;L^{\f{6p}{p+7}}(\mathbb{T}^{3}))}   \| \omega\|_{ L^{p}(0,T;L^{\f{6p}{5p-7}}(\mathbb{T}^{3}))}.
\ea\end{equation}
 and
$$ \|v\|_{L^{\f{ p}{ p-2}}(0,T;L^{\f{3p}{7-2p}}(\mathbb{T}^{3}))}\leq
C\|v\|^{\f{21-7p}{7}}_{L^{\infty}(0,T;L^{2}(\mathbb{T}^{3}))}
\|\nabla v\|^{\f{ 7p-14}{7}}_{ L^{p}(0,T;L^{\f{6p}{5p-7}}(\mathbb{T}^{3}))}+C
 \|v\|_{L^{\infty}(0,T;L^{2}(\mathbb{T}^{3}))}.$$
 The H\"older inequality ensures that
$$ \|\omega\times v \|_{ L^{\f{p}{p-1}}(0,T;L^{\f{6p}{p+7}}(\mathbb{T}^{3}))}
 \leq  C\|v\|_{L^{\f{ p}{ p-2}}(0,T;L^{\f{3p}{7-2p}}(\mathbb{T}^{3}))} \|\nabla v\|_{ L^{p}(0,T;L^{\f{6p}{5p-7}}(\mathbb{T}^{3}))}.$$
 We can apply Lemma \ref{lem2.4}  to complete the proof of this part.

 (4)
Choosing   $p_{2}(2-\theta)=\infty,q_{2}(2-\theta)=2$, $q=\f{q}{q-1}$ and $p=\f{p}{p-1}$ in Lemma  \ref{lem2.3}, we get $\theta p_{1}=\f{p\theta}{p-1}$ and $\theta q_{1}=\f{2 q\theta}{q\theta-2}$
and
\be\label{4.17}
\|(v^{\varepsilon}\otimes v^{\varepsilon}) -(v\otimes v)^{\varepsilon}\|_{L^{\f{p}{p-1}}(0,T;L^{\f{q}{q-1}}(\mathbb{T}^3))}
 \leq o(\varepsilon^{\theta\alpha}),\ee
 where  $\theta p_{1}= k$ and $\theta q_{1}=\ell$
  were used.

 It follows from Lemma \ref{lem2.2} that
\be \label{4.18}\|\nabla \omega^{\varepsilon}\|_{L^{p}(0,T;L^{q}(\mathbb{T}^3))}
 \leq O(\varepsilon^{\beta-1}).\ee
Inserting \eqref{4.17} and \eqref{4.18} into \eqref{4.16}, we arrive at
 $$\ba
&\B|\int_{0}^{T}\int_{\mathbb{T}^3}
	\big((v\otimes v)^{\varepsilon}-
	(v^{\varepsilon}\otimes v^{\varepsilon})\big):\nabla v^{\varepsilon}dxdt\B|
\leq o(\varepsilon^{ \theta\alpha+\beta-1}).
\ea$$
This ensures that the right-hand side of \eqref{4.12} vanishes as $\varepsilon\rightarrow0$. The proof of this part is finished.
(5) Exchanging $o$ and $O$  in \eqref{4.17}-\eqref{4.18}, we can prove this theorem.
\end{proof}

\section{Conclusion}
A celebrated result in the study of Onsager conjecture is that
 $v\in L^{3} (0,T; B^{1/3}_{3,c(\mathbb{N})}(\mathbb{R}^3)) $ guarantees the energy conservation of the Euler equations on the whole space, which is shown by  Cheskidov-Constantin-Friedlander-Shvydkoy in  \cite{[CCFS]}. Inspired by recent works \cite{[BG],[B]} of
 Berselli and  Berselli-Georgiadis, we extend it to a more general case $v\in L^{p}(0,T;B^{\f1p}_{\f{2p}{p-1},c(\mathbb{N})} )$ for $1\leq p\leq3. $ We provide two different methods to show it. The first is interpolation technique and the second one is Littlewood-Paley theory as \cite{[CCFS]} for the case $1< p\leq3. $ It seems that the first one is better than the second one. As a byproduct, a  observation is that
the general case immediately implies    energy conservation sufficient condition in terms of gradient  in  space $L^{p}(0,T; L^{\f{6p}{5p-5}}(\mathbb{R}^{3}))$  derived in \cite{[BC],[WMH],[Zhang],[BY2],[Berselli]} for the Navier-Stokes equations.
In addition, parallel to energy,  the helicity conservation of weak solutions in inviscid flow is also considered. It seems that our argument can be applied to other fluid equations.

It is shown that   the space $B^{1/3}_{3,c(\mathbb{N})}$ for the energy conservation  are sharp in   \cite{[CCFS]}. It is an interesting problem  to construct an example to illustrate the sharpness of space $B^{\f1p}_{\f{2p}{p-1},c(\mathbb{N})} $ for $1\leq p\leq3.$
\section*{Acknowledgement}

 Wang was partially supported by  the National Natural
 Science Foundation of China under grant (No. 11971446  and No. 12071113)    and sponsored by Natural Science Foundation of
Henan    (No. 232300421077).
Wei was partially supported by the National Natural Science
Foundation of China under grant (No. 11601423 and No. 12271433).
Wu was partially supported by the National Natural Science
Foundation of China under grant No. 11771423.
Ye was partially sponsored by National Natural Science Foundation of
China (No. 11701145) and Natural Science Foundation of Henan   (No. 232300420111).



\end{document}